\documentclass{amsart}

\usepackage{amsmath}
\usepackage{amssymb}
\usepackage{enumerate}

\usepackage[matrix,arrow,arc]{xy}

\renewcommand{\mod}{\operatorname{mod}}
\newcommand{\Hom}{\operatorname{Hom}}
\newcommand{\End}{\operatorname{End}}

\newcommand{\ann}{\operatorname{ann}}
\newcommand{\ind}{\operatorname{ind}}

\newcommand{\soc}{\operatorname{soc}}
\newcommand{\rad}{\operatorname{rad}}

\newcommand{\Tr}{\operatorname{Tr}}

\newcommand{\umod}{\operatorname{\underline{mod}}}

\newcommand{\tp}{\operatorname{top}}

\newcommand{\op}{\operatorname{op}}
\newcommand{\Ker}{\operatorname{Ker}}

\newcommand{\charact}{\operatorname{char}}

\newcommand{\cC}{\mathcal{C}}

\newcommand{\bA}{\mathbb{A}}
\newcommand{\bB}{\mathbb{B}}
\newcommand{\bC}{\mathbb{C}}
\newcommand{\bD}{\mathbb{D}}
\newcommand{\bE}{\mathbb{E}}
\newcommand{\bF}{\mathbb{F}}
\newcommand{\bG}{\mathbb{G}}
\newcommand{\bZ}{\mathbb{Z}}

\newcommand{\wh}{\widehat}

\newtheorem{theorem}{Theorem}[section] 
\newtheorem{lemma}[theorem]{Lemma}     
\newtheorem{corollary}[theorem]{Corollary}
\newtheorem{proposition}[theorem]{Proposition}

\newtheorem*{problem*}{Problem}

\theoremstyle{definition}

\newtheorem{example}[theorem]{Example}
\newtheorem{definition}[theorem]{Definition}

\title[Socle deformations of selfinjective orbit algebras]%
 {Socle deformations of selfinjective orbit algebras of tilted type} 

\author[A. Skowro\'nski and K. Yamagata]{Andrzej Skowro\'nski and Kunio Yamagata} 

	\dedicatory{Dedicated to Jos\'e Antonio de la Pe\~{n}a on the occasion of his 60th birthday}

\begin{document}
\maketitle

\dedicatory

\begin{abstract}
We survey recent development of the study 
of finite-dimensional selfinjective algebras 
over a field which are socle equivalent to 
selfinjective orbit algebras of tilted type.
\end{abstract}

\section{Introduction} 
\label{sec:intro}

\noindent

Throughout this article, by an algebra we mean a basic, connected, finite-dimensional 
associative $K$-algebra with identity over a fixed field $K$.
For an algebra $A$, we 
denote by $\mod A$ the category of finite-dimensional right $A$-modules 
and $\ind A$ the full subcategory of $\mod A$ of indecomposable modules.
We denote by $\Gamma_A$ the Auslander-Reiten quiver of $A$,
and by $\tau_A$ and $\tau_A^{-1}$ the Auslander-Reiten 
translations $D \Tr$ and $\Tr D$, respectively, where 
$D$ is the standard duality $\Hom_K(-,K)$ on $\mod A$ and $\Tr$ is the 
transpose.

An algebra $A$ is called \emph{selfinjective}
if $A$ is injective in $\mod A$, or equivalently
the projective modules in $\mod A$ are injective.
Two selfinjective algebras $A$ and $A'$ are said
to be \emph{socle equivalent} if the quotient algebras
$A/\soc(A)$ and $A' / \soc (A')$ are isomorphic.
In this case, we also call $A$ a \textit{socle deformation} of $A'$.
Moreover, two selfinjective algebras $A$ and $A'$ are called 
\textit{stably equivalent} if their stable categories $\umod A$ and $\umod A'$
are equivalent as $K$-categories.

In the representation theory of selfinjective algebras an important
role is played by the selfinjective algebras $A$ which admit
Galois covering of the form $\widehat{B} \to \widehat{B}/G = A$,
where $\widehat{B}$ is the repetitive category of an algebra $B$
and $G$ is an admissible group of automorphisms of $\widehat{B}$.
In this theory, the selfinjective orbit algebras $\widehat{B}/G$
given by algebras $B$ of finite global dimension 
and infinite cyclic groups $G$ are of special interest.
Namely, frequently interesting selfinjective algebras
are Morita equivalent to socle deformations of such orbit algebras,
and we may reduce their representation theory to that
for the corresponding algebras of finite global dimension.
For example, it is the case for selfinjective algebras
of polynomial growth over an algebraically closed field $K$
(see Section~\ref{sec:2} for details).
This is also the case for the tame principal blocks 
of the enveloping algebras of restricted Lie algebras \cite{FS1}, 
or more general infinitesimal groups \cite{FS2},
over algebraically closed fields of odd characteristic.
We also note that the prominent class of special biserial
algebras over an algebraically closed field is formed by
the orbit algebras of repetitive categories 
(see \cite{DS}, \cite{ES}, \cite{PS}).

For an arbitrary field $K$, even the structure of selfinjective algebras
of finite representation type over $K$
is far from being understood 
(see Section~\ref{sec:6} for new results toward solution
of this problem).
In this article, we survey old and new results concerning selfinjective
algebras $A$ over a field which are socle equivalent to orbit algebras
$\widehat{B}/G$ for an algebra $B$ and $G$ the infinite cyclic group
generated by the composition $\varphi \nu_{\widehat{B}}$
of the Nakayama automorphism $\nu_{\widehat{B}}$
and a positive automorphism $\varphi$ of $\widehat{B}$.
In particular, we present, in  Sections \ref{sec:3} and \ref{sec:4},
criteria for a selfinjective algebra $A$ to have this property.
Sections \ref{sec:5} and \ref{sec:6} are devoted to presentations
of interesting selfinjective algebras as socle deformations 
of orbit algebras of repetitive categories of tilted algebras.
In the final  Section~\ref{sec:7} we show that, for the class 
of selfinjective algebras of tilted type, discussed in this article,
the stable quivalence and socle equivalence coincide.

For basic background on the representation theory discussed in this article
we refer to the books
\cite{ASS}, \cite{SY9}, \cite{SY11},
and the survey articles \cite{S4}, \cite{SY8}.

 \section{Repetitive algebras and orbit algebras}
 \label{sec:2}

Let $B$ be an algebra and $1_{B}=e_{1}+\cdots +e_{n}$ 
a decomposition of the identity of $B$ into a sum of 
pairwise orthogonal primitive idempotents  $e_{1}, \dots , e_{n}$ of $B$. 
The algebra is regarded as a $K$-category, denoted by $B$ again, whose objects
are $e_1, \dots, e_n$ and the morphism space $B(e_i, e_j) 
= \Hom_B(e_i, e_j)$ is $\Hom_B(e_iB, e_jB)$ for $i, j \in \{1, \dots, n\}$.
We then associate to $B$  a selfinjective locally bounded 
$K$-category $\widehat B$, called the 
\emph{repetitive category} of $B$ (see~\cite{HW}). 
The objects of $\widehat B$ are $e_{m,i}$ for $m\in{\mathbb{Z}}$,
$i\in \{1, \dots, n\}$, and the morphism spaces are defined as follows
\[
\widehat B(e_{m,i},e_{r,j})=\left\{\begin{array}{ll}
e_{j}Be_{i},    & r=m,\\
D(e_{i}Be_{j}),& r=m+1,\\
0,& \textrm{otherwise}.
\end{array} \right.
\]
Observe that 
$e_{j}Be_{i}=\Hom _{B}(e_{i}B,e_{j}B)$, $D(e_{i}Be_{j})=e_{j}D(B)e_{i}$ 
and
\[
\bigoplus_{(m,i)\in{\mathbb{Z}\times \{1, \dots ,n\}}} \widehat B(e_{m,i},e_{r,j})
=e_{j}B \oplus D(Be_{j}),
\]
for any $r\in{\mathbb{Z}}$ and $j\in\{1, \dots ,n\}$.
The composition of morphisms in $\wh{B}$ 
is naturally defined by the multiplication in the algebra $B$
and the $B$-bimodule structure of $D(B)$.

An automorphism $\varphi$ of the $K$-category $\widehat{B}$ is said
to be \textit{positive} or \textit{rigid}, respectively, if for each $(m,i)\in 
\bZ \times \{1, \dots, n\}$, 
\begin{center}
	$\varphi(e_{m,i})=e_{p,j}$ for some $p\geq m$ and $j\in \{1, \dots, n\}$,
\end{center}
or
\begin{center}
	$\varphi(e_{m,i})=e_{m,j}$ for some $j\in \{1, \dots, n\}$.
\end{center}
A \textit{strictly positive} automorphism of $\widehat B$ is a positive but not rigid automorphism
of $\widehat{B}$.
In particular, the \textit{Nakayama automorphism} of $\widehat B$, 
denoted by $\nu_{\widehat B}$, is the strictly positive automorphism defined by 
\[
\nu_{\widehat B}(e_{m,i})=e_{m+1,i} \quad \textrm{for all} \quad (m,i)\in \mathbb{Z}\times\{1, \dots ,n\}.
\]

A group $G$ of automorphisms  of $\widehat B$ is said to be 
\emph{admissible} if $G$ acts freely on the set of objects of 
$\widehat B$ and has finitely many orbits.
Following P.~Gabriel \cite{G}, we may consider 
the orbit category $\widehat B/G$ of $\widehat B$ with respect 
to $G$, whose objects are the $G$-orbits of objects in $\widehat B$.
The morphism space $(\widehat{B}/G)(a,b)$ for 
objects $a, b$ in $\widehat{B}/G$ is the subspace of
$\prod_{(x,y)\in{a\times b}} \widehat B(x,y)$ consisting of those elements
$ (f_{y,x})$ satisfying 
\[
  gf_{y,x}=f_{gy,gx} \,\,\,\, \text{for all}\,\,  g\in{G}, (x,y)\in{a\times b}.
\]
By the definition $\widehat B/G$ has
finitely many objects and the morphism spaces are finite-dimensional.
It therefore associates naturally the finite-dimensional 
selfinjective $K$-algebra $\bigoplus (\widehat B/G)$ which is the
direct sum of all morphism spaces in $\widehat B/G$, called the
\emph{orbit algebra} of $\widehat B$ with respect to $G$. 
A typical example is the trivial extension algebra $B\ltimes D(B)$ of 
an algebra $B$, which is isomorphic to the orbit algebra 
$\widehat B/(\nu_{\widehat{B}})$.
More generally, the $r$-\emph{fold trivial extension algebra} $T(B)^{(r)}$ 
of $B$ is defined as the orbit algebra 
$\widehat B/(\nu_{\widehat{B}}^{r})$, for $r \geq 1$. 
Note that $T(B)^{(1)}= T(B)$.

A selfinjective algebra $A$ is said to be \textit{of tilted type} if $A$ is 
isomorphic to an orbit algebra $\widehat{B}/G$ for a tilted algebra $B$
and an admissible  automorphism group $G$ of $\widehat{B}$. 
An important remark is that in this case 
any admissible automorphism group of $\wh{B}$ is
an infinite cyclic group generated by 
a strictly positive automorphism of $\wh{B}$
(see \cite[Theorem 7.1]{SY8}).
In the paper, by a \emph{tilted algebra} we mean the endomorphism
algebra $B = \End_H(T)$ of a tilting module $T$ in the module
category $\mod H$ of a basic, connected, hereditary algebra $H$ 
over a field $K$.

\medskip

The classification of all representation-finite selfinjective algebras 
over an algebraically closed field was given in the early 1980's by C. Riedtmann
(see \cite{BLR}, \cite{Ried}, \cite{Rd2}, \cite{Rd3}) 
via a combinatorial classification of the Auslander-Reiten
quivers of these algebras, based on the fundamental paper \cite{Ried}
linking them to the Dynkin quivers.
Equivalently, the Riedtmann's classification can be presented in terms of the
orbit algebras as follows (see \cite[Section~3]{S4}).

\begin{theorem}
\label{th:2.1}
 Let $A$ be a non-simple, basic, connected, selfinjective algebra 
over an algebraically closed field $K$.
 Then $A$ is of finite representation type if and only if
 $A$ is socle equivalent to an orbit algebra $\wh{B}/(\varphi)$, where $B$ is 
 a tilted  algebra of Dynkin type and $\varphi$ is a strictly  positive 
 automorphism of $\wh{B}$.
 \end{theorem}
 
In the theorem, we may replace ``socle equivalence'' by ``isomorphism''
if $\charact(K) \neq 2$, but it is not the case in general
(see \cite{Rd3}  and \cite[Section~3]{S4}). 

We would like to stress that a crucial role in proving
the above interpretation of the Riedtmann's classification
theorem is played by the Galois covering techniques introduced
by P. Gabriel in \cite{G} and the description of the module categories
of repetitive categories of tilted algebras of Dynkin type
given by D.~Hughes and J.~Waschb\"usch in \cite{HW}.
This was the starting point for the study
of representation-infinite selfinjective algebras
of polynomial growth in \cite{S3},
where it was shown that all these algebras,
having simply connected Galois coverings,
are related to tilted algebras of Euclidean type
and tubular algebras.
Here, the new results on Galois coverings 
of representation-infinite algebras proved by
P.~Dowbor and A.~Skowro\'nski in \cite{DS1}, \cite{DS}
are heavily applied.

The classification of arbitrary representation-infinite
domestic selfinjective algebras was completed in
the series of papers
\cite{BoS1}, \cite{BoS2}, \cite{BoS3}, \cite{LS1} 
(see \cite[Section~4]{S4}).
In particular, we have the following theorem.

\begin{theorem}
\label{th:2.2}
Let $A$ be a basic, connected, selfinjective algebra 
over an algebraically closed field $K$.
Then $A$ is representation-infinite domestic if and only if 
$A$ is socle equivalent to an orbit algebra $\wh{B}/(\varphi)$, 
where $B$ is a tilted algebra of Euclidean type and 
$\varphi$ is a strictly positive automorphism of $\wh{B}$.
\end{theorem}

We would like to stress that, for any algebraically field $K$,
there are representation-infinite domestic selfinjective
algebras which are not orbit algebras of repetitive categories
of tilted algebras of Euclidean type 
(see \cite{BoS3} for description of these algebras).
We also point that the structure of the module categories
$\mod A$ of the representation-infinite domestic selfinjective
algebras $A$ follows from Theorem~\ref{th:2.2}
and the results proved in \cite{ANS} and \cite{S3}.

For an algebra $A$, the \emph{infinite radical} $\rad_A^{\infty}$ 
of $\mod A$ is the intersection of all powers $\rad_A^i$, $i \geq 1$,
of the radical $\rad_A$ of $\mod A$.
We note that, by a classical result of M.~Auslander, $\rad_A^{\infty} = 0$
if and only if $A$ is representation-finite.
The following theorem is a consequence of Theorem~\ref{th:2.2}
and the main result proved by O.~Kerner and A.~Skowro\'nski
in \cite[Theorem]{KS}.

\begin{theorem}
\label{th:2.3}
Let $A$ be a basic, connected, representation-infinite 
selfinjective algebra over an algebraically closed field $K$.
Then the infinite radical $\rad_A^{\infty}$  of $\mod A$
is nilpotent if and only if 
$A$ is socle equivalent to an orbit algebra $\wh{B}/(\varphi)$, 
where $B$ is a tilted algebra of Euclidean type and 
$\varphi$ is a strictly positive automorphism of $\wh{B}$.
\end{theorem}

The classification of arbitrary non-domestic selfinjective algebras 
of polynomial growth was completed in the series of papers
\cite{BiS1}, \cite{BiS2}, \cite{BiS3}, \cite{LS2} 
(see \cite[Section~5]{S4} for details).
In particular, we have the following theorem.

\begin{theorem}
\label{th:2.4}
Let $A$ be a basic, connected, selfinjective algebra 
over an algebraically closed field $K$.
Then $A$ is non-domestic of polynomial growth if and only if 
$A$ is socle equivalent to an orbit algebra $\wh{B}/(\varphi)$, 
where $B$ is a tubular algebra and 
$\varphi$ is a strictly positive automorphism of $\wh{B}$.
\end{theorem}

In the theorem, we may replace ``socle equivalence'' by ``isomorphism''
if $\charact(K)$ is different from $2$ and $3$ (see \cite{BiS3} for details).
We also mention that the structure of the module categories
$\mod A$ of non-domestic selfinjective algebras $A$ of
polynomial growth follows from Theorem~\ref{th:2.4}
and the results proved in \cite{NS} and \cite{S3}.

Recall that an algebra $A$ is called \emph{periodic}
if it is periodic with respect to action of the syzygy operator $\Omega_{A^e}$
in the module category $\mod A^e$ over the enveloping algebra
$A^e = A^{\op} \otimes_K A$.
It is known that if $A$ is periodic then $A$ is selfinjective
and all nonprojective indecomposable modules in $\mod A$
are periodic with respect to action of the syzygy operator
$\Omega_{A}$ in $\mod A$.
It is known that every non-simple, basic, connected, representation-finite 
selfinjective algebra $A$ is periodic, by a theorem proved
by A.~Dugas in \cite{Du}. 
The following theorem is a consequence of Theorem~\ref{th:2.4}
and the main result proved by 
J. Bia\l kowski, K. Erdmann and A. Skowro\'nski
in \cite[Theorem~1.1]{BES}.

\begin{theorem}
\label{th:2.5}
Let $A$ be a basic, connected, representation-infinite 
selfinjective algebra of polynomial growth 
over an algebraically closed field $K$.
Then $A$ is periodic if and only if 
$A$ is socle equivalent to an orbit algebra $\wh{B}/(\varphi)$, 
where $B$ is a tubular algebra and 
$\varphi$ is a strictly positive automorphism of $\wh{B}$.
\end{theorem}

We also note that the algebras $B$ occurring in the above 
theorems are algebras of global dimension $1$ or $2$.


Recently, Theorems \ref{th:2.1}, \ref{th:2.2} and \ref{th:2.4}
have been applied by 
S. Ariki, R. Kase, K. Miyamoto and K. Wada
to provide in \cite{AKMW}
a complete classification of selfinjective 
cellular algebras of polynomial growth,
over algebraically closed fields $K$
with $\charact(K) \neq 2$.


We refer also to \cite{EKS} (respectively, \cite{LS3}) for the
representation theory and homological properties of the
orbit algebras $\widehat{B}/G$ of tilted algebras $B$ 
of wild type (respectively, quasi-tilted algebras $B$ of wild
canonical type).
These algebras $B$ are also of global dimension $1$ or $2$.

In general, the following problem arises naturally.

\begin{problem*}
Describe the basic, connected, selfinjective algebras 
over an arbitrary field $K$ which are socle
equivalent to an orbit algebra $\wh{B}/G$, 
where $B$ is an algebra of finite global dimension and 
$G$ is an infinite cyclic group generated by 
a strictly positive automorphism of $\wh{B}$.
\end{problem*}

This seems to be a very hard problem in general.
In the article, we present recent results concerning the above problem 
in the following case (coming naturally in several considerations):

\smallskip

\begin{itemize}
 \item
  \textit{$B$ is a tilted algebra and $G$ is generated by
  $\varphi\nu_{\widehat{B}}$
  for a positive automorphism $\varphi$ of $\wh{B}$.}
\end{itemize}

\smallskip

To illustrate the situation, let us consider the following simple example.

\begin{example}
Let $Q_n$ be the cyclic quiver 
\[
	1 \xrightarrow{\;\;\alpha_1\;\;} 2 \xrightarrow{\;\;\alpha_2\;\;} \cdots
	\xrightarrow{\alpha_{n-1}} n \xrightarrow{\;\;\alpha_n\;\;} 1, \;\; n>1,
\]
and $A_n=KQ_n/N$, where $KQ_n$ is the path algebra and $N$ is the ideal
generated by the compositions of all consecutive three arrows.
Let $A=A_n$ for simplicity. 
Then $A$ is a selfinjective Nakayama algebra with $\rad^3(A) =0$.
Take the right module $M=M_1\oplus M_2$ where $M_1=e_1A/\rad^2(e_1A)$
and $M_2=e_2A/\rad e_2A$, and the right annihilator $I=r_A(M)$ of
$M$ in $A$. Then 
\[
	I=\rad^2 (e_1A) \oplus \rad (e_2A) \oplus e_3A \oplus 
	\cdots \oplus e_nA.
\]
 The factor algebra $B=A/I$ is isomorphic to the 
path algebra of the quiver $\Delta$ of the form
\[
	1 \xrightarrow{\;\;\alpha_1\;\;} 2
\]
of Dynkin type $\mathbb{A}_2$.
 It is easily seen
that $A\cong \wh{B}/(\varphi\nu_{\wh{B}})$ for a positive automorphism
$\varphi$ of $\wh{B}$. 
For example, in the case $n=5$, we have 
$\varphi=\sigma\nu_{\wh{B}}$ where 
$\sigma$ is an automorphism satisfying 
$\sigma^2=\nu_{\wh{B}}$. 

Thus the selfinjective algebra $A_n$ with arbitrarily $n$ many simple modules
is always recovered from the path algebra $B=K\Delta$ and two
indecomposable modules $M_1$, $M_2$ lying on a slice of $\Gamma_{A_n}$.	
\end{example}

\section{Criterion theorems}
\label{sec:3}

We recall ring theoretical criterion theorems for a selfinjective algebra $A$
to be socle equivalent to an orbit algebra of tilted type. 
Throughout this section, $A$ denotes a selfinjective algebra and 
$\{e_1, \dots, e_r\}$ is a set of pairwise orthogonal primitive idempotents of $A$
with $1_A=e_1+ \cdots + e_r$.

For an algebra $\Lambda$ and a subset $X$ of a right or left $\Lambda$-module 
$M$, by we denote the right or left annihilator of $X$ in $\Lambda$ respectively,
\[
r_{\Lambda}(X)=\{\lambda \in \Lambda \mid X\lambda = 0\},
\quad
l_{\Lambda}(X)=\{\lambda \in \Lambda \mid \lambda X = 0\}.
\]

Let $I$ be an ideal of $A$ and $B=A/I$ the factor algebra. 
We can take an idempotent $e$ of $A$ such that $e+I$ is the identity 
of $B$, $e=e_1+ \cdots + e_n$ $(n \leq r)$ and $\{e_1, \dots, e_n\}$
is the set of all idempotents in $\{e_i \mid 1\leq i\leq r\}$ not contained in $I$.
Thus $e_1+I, \dots, e_n+I$ are pairwise orthogonal primitive idempotents of $B$
and $1_B= \bar{e}_1 + \cdots + \bar{e}_n$, where $\bar{e}_i=e_i+I \in B$
for $i\in \{1, \dots, n\}$.
By Krull-Schmidt theorem, such an idempotent $e$ is uniquely determined up to 
inner automorphism of $A$, and is called a \textit{residual identity} of 
$B=A/I$.

By a Nakayama's theorem \cite[Theorem IV.6.10]{SY9} 
the annihilator operations $l_A$ and $r_A$
induce a Galois correspondence between the lattice $\mathcal{R}_A$ of 
right ideals and the lattice $\mathcal{L}_A$ of left ideals of $A$:
$l_A: \mathcal{R}_A \to \mathcal{L}_A$,
$r_A: \mathcal{L}_A \to \mathcal{R}_A$,
each of which is the inverse to the other.
An important consequence of this Galois correspondence is the following 
statement on the residual identity
(see \cite[Proposition~2.3]{SY1} and \cite[Lemma~5.1]{SY6}).
\begin{lemma} \label{residual identity}
	Let $A$ be a selfinjective algebra and $I$ an ideal of $A$.
	Then an idempotent $e$ of $A$ is a residual identity of $B=A/I$ 
	if $\l_A(I)=eI$ or $r_A(I)=Ie$.
	Moreover, in this case, $\soc A \subseteq I$ and
	$l_{eAe}(I)=eIe=r_{eAe}(I)$.
\end{lemma}
Observe that the canonical correspondence $B \to eAe/eIe$, 
$a+I \mapsto eae+eIe$, is an algebra isomorphism. 

In the following first criterion theorem, the implication 
(ii) $\Rightarrow$ (i) was proved in \cite{SY3} and 
the converse in  \cite{SY6}.
\begin{theorem} \label{criterion 1}
	For a basic, connected, selfinjective algebra $A$
	over a field $K$, the following
	statements are equivalent:
	\begin{enumerate}[\upshape(i)]
		\item $A$ is isomorphic to an orbit algebra $\wh{B}/(\varphi\nu_{\wh{B}})$
			for an algebra $B$ and a positive automorphism $\varphi$ of $\wh{B}$.
		\item There is an ideal of $I$ of $A$ and an idempotent $e$ of $A$
		such that
		\begin{enumerate}[\upshape(a)]
			\item $r_A(I)=eI$,
			\item The canonical algebra homomorphism
			\[
			\rho : eAe \to eAe/eIe,\;eae \mapsto eae+eIe,
			\] is a retraction,
			that is, there is an algebra homomorphism $\eta: eAe/eIe \to eAe$ with
			$\rho\eta=id_{eAe/eIe}$.
		\end{enumerate}
	\end{enumerate}
\end{theorem}

It should be noticed that the statement (ii)(b) holds always in case $K$ is
algebraically closed. Under the statement (ii), $B$ in (i) may be chosen
as $A/I$, and the idempotent $e$ in (ii) is determined by $I$ as a residual
identity of $A/I$, see Lemma \ref{residual identity}.

It was observed that in many important situations we may
replace a selfinjective algebra $A$ by its socle deformation
satysfying the statement (ii)(b).
In order to explain this we need the notion of 
deforming ideal  from \cite{SY1}.

\begin{definition}
\label{deforming ideal}
	 Let $A$ be a selfinjective algebra, $I$ an ideal of $A$, and $e$ 
	 a residual identity of $A/I$. 
	 Then $I$ is said to be a \textit{deforming ideal}
	 of $A$ if the following conditions are satisfied:
	 \begin{flalign*}
	 	  (\textup{D1})& \quad 	l_{eAe}(I)=eIe=r_{eAe}(I), \hspace*{7cm}\\
	     (\textup{D2})&\quad  \text{the valued quiver} \; Q_{A/I} \;\text{of $A/I$ is acyclic}.
	 \end{flalign*}
\end{definition}

The following statement is an immediate consequence of 
Lemma~\ref{residual identity} and Definition~\ref{deforming ideal},
and is important 
for further considerations.

\begin{proposition}\label{prop.def.ideal}
	Let $A$ be a selfinjective algebra, $I$ an ideal of $A$ and $B=A/I$.
	Assume that 
	\begin{enumerate}[\upshape (i)]
		\item $r_{A}(I)=eI$ for an idempotent $e$ of $A$,
		\item the valued quiver $Q_{B}$ of $B$ is acyclic.
	\end{enumerate}
	Then 
	$I$ is a deforming ideal of $A$ and $e$ is a residual identity of $B$.
\end{proposition}

Now assume that $I$ is a deforming ideal of 
a selfinjective algebra $A$. 
Then the canonical correspondence
$eAe/eIe\to A/I$, $eae+I \mapsto eae+I$, is an algebra isomorphism,
which makes $I$ as an $eAe/eIe$-bimodule and allows us to define
a new algebra, denoted by $A[I]$, as follows.
Let $A[I]$ be the direct sum of $K$-vector spaces 
$(eAe/eIe)\oplus I$ and define the multiplication 
in $A[I]$ by the rule
\begin{displaymath}
(b,x)\cdot (c,y)=(bc,by+xc+xy)
\end{displaymath}
for $b,c\in{eAe/eIe}$ and $x,y\in I$. 
Then $A[I]$ is actually a $K$-algebra with the identity $(e+eIe,1_{A}-e)$
and the ideal $\{(0, x)\mid x\in I\}$.
By identifying $x\in{I}$ with $(0,x)\in{A[I]}$,
we may regard $I$ as an ideal of $A[I]$, so that $A[I]/I$ denotes
the factor algebra of $A[I]$ by $I$ with the residual identity 
$e=(e+eIe,0)$. 
Thus, by identifying $e \in A$ with $(e,0)\in A[I]$,
we have 
\[
A[I]/I=eAe/eIe\cong A/I, \quad
eA[I]e=(eAe/eIe)\oplus eIe.
\]
Moreover, the canonical algebra 
epimorphism $eA[I]e\to eA[I]e/eIe$ is a retraction, which is obvious from
the definition but the main reason why the algebra $A[I]$ has been introduced.
In fact, $A[I]$ keeps other important properties of $A$ as shown in the next
theorem established in \cite[Theorem~4.1]{SY1}, \cite[Theorem~3]{SY2} 
and \cite[Lemma~3.1]{SY7}.

\begin{theorem}
	\label{thm:2.4}
	Let $A$ be a selfinjective algebra over a field $K$ and 	$I$ a deforming ideal 
	of $A$. 	Then the following statements hold.
	\begin{enumerate}[\upshape (i)]
		\item $A[I]$ is a selfinjective algebra with the same Nakayama permutation 
		as $A$ and $I$ is a deforming ideal of $A[I]$.
		\item $A$ and $A[I]$ are socle equivalent.
		\item $A$ and $A[I]$ are stably equivalent.
		\item $A[I]$ is a symmetric algebra if $A$ is a symmetric algebra.
	\end{enumerate}
\end{theorem}

We note here that a socle equivalence does not imply a stably equivalence in general, 
as pointed out by J. Rickard (see \cite{O},  \cite{SY9}). Moreover, 
a selfinjective algebra $A$ with a deforming ideal $I$ is not always 
isomorphic to $A[I]$ (see~\cite[Example~4.2]{SY3}), but this is the case when
$K$ is an algebraically closed field  (\cite[Theorem~3.2]{SY1}).
\medskip

The following theorem proved in~\cite[Theorem~4.1]{SY3} 
shows the importance of the algebras  $A[I]$.

\begin{theorem}
	\label{thm:2.5}
	Let $A$ be a selfinjective algebra, $I$ an ideal of $A$, 
	$B=A/I$ and $e$ an idempotent of $A$. 
	Assume that $r_{A}(I)=eI$ and $Q_{B}$ is acyclic. 
	Then the following statements are true:
	\begin{enumerate}[\upshape (i)]
		\item 	$A[I]$ is isomorphic to an orbit algebra 
		$\widehat B/(\varphi\nu_{\widehat B})$ for some positive 
		automorphism $\varphi$ of $\widehat B$.
		\item $A$ is a socle deformation of 	$\widehat B/(\varphi\nu_{\widehat B})$ 
		for some positive automorphism $\varphi$ of $\widehat B$.
   \end{enumerate}
\end{theorem}

\begin{proof}
	First note that, by Proposition~\ref{prop.def.ideal}, the ideal $I$ is 
	a deforming ideal of $A$, so $A[I]$ is well defined. 
	
	(ii) follows from (i) and
	Theorem \ref{thm:2.4} (ii). To show (i),  we apply Theorem \ref{criterion 1} 
	to $A[I]$ and its ideal $I$.
	In fact, it is seen that $r_{A[I]}(I)=eI$ in $A[I]$, and the canonical
	algebra homomorphism $eA[I]e \to eA[I]/eIe$ is a retraction.
	As a result, the conditions (a) and (b) in (ii) of Theorem \ref{criterion 1}
	are satisfied, which ensures the existence of an algebra isomorphism
	$A[I] \to \wh{B}/(\varphi\nu_{\wh{B}})$ for a positive automorphism
	$\varphi$ of $\wh{B}$. 
\end{proof}

In the theorem, the algebra $A$ is not necessarily isomorphic to 
an orbit algebra $\widehat{B}/(\varphi\nu_{\widehat{B}})$, 
where $B$ is an algebra and $\varphi$ is a positive automorphism 
of $\widehat{B}$ (see~\cite[Proposition~4]{SY4}).

The following result proved in~\cite[Proposition~3.2]{SY5} 
describes a situation when the algebras $A$ and $A[I]$ are isomorphic.

\begin{theorem}
\label{thm:3.7}
Let $A$ be a selfinjective algebra 
over a field $K$,
having a deforming ideal $I$, $B=A/I$,
$e$ be a residual identity of $B$, 
and $\nu$ the Nakayama permutation of $A$. 
Assume that $IeI=0$ and $e_{i}\neq e_{\nu(i)}$, 
for any primitive summand $e_{i}$ of $e$. 
Then the algebras $A$ and $A[I]$ are isomorphic. 
In particular, $A$ is isomorphic to an orbit algebra 
$\widehat B/(\varphi\nu_{\widehat B})$ for some 
positive automorphism $\varphi$ of $\widehat B$.
\end{theorem}

\section{Hereditary stable slice}
\label{sec:4}

In this section we explain a new characterization from \cite{SY12}
of the socle deformations
of selfinjective orbit algebras $\wh{B}/G$ of tilted type where $G$ is an
admissible group generated by $\varphi\nu_{\wh{B}}$ for a positive 
automorphism $\varphi$ of $\wh{B}$.

Let $A$ be a selfinjective algebra over a field $K$ and 
$\Gamma_A^s$ the stable Auslander-Reiten quiver of $A$, obtained from
$\Gamma_A$  by removing the projective modules and the arrows 
attached to them. 
By $\ind \mathcal{P}(A)$ we understand 
the family of indecomposable projective modules in $\mod A$.

Following \cite{SY10}, 
a full valued subquiver $\Delta$ of $\Gamma_A$ is said to be a
\emph{stable slice} if the following conditions are satisfied:
\begin{enumerate}[\upshape (1)]
	\item
	$\Delta$ is connected, acyclic, and without projective modules.
	\item
	For any valued arrow $V \xrightarrow{(a,a')} U$ in $\Gamma_A$
	with $U$ in $\Delta$ and $V$ non-projective,
	$V$ belongs to $\Delta$ or to $\tau_A\Delta$.
	\item
	For any valued arrow $U \xrightarrow{(b,b')} V$ in $\Gamma_A$
	with $U$ in $\Delta$ and $V$ non-projective,
	$V$ belongs to $\Delta$ or to $\tau_A^{-1}\Delta$.
\end{enumerate}

A stable slice $\Delta$ of $\Gamma_A$ is said to be \textit{right regular}
(respectively, \textit{left regular}) if $\Delta$ does not contain the radical
$\rad P$ (respectively, the socle factor $P/\soc P$) of any $P$ from 
$\ind \mathcal{P}(A)$.
A stable slice $\Delta$ is said to be \textit{almost right regular} if for any $P$
from $\ind \mathcal{P}(A)$ with $\rad P$ lying on $\Delta$, $\rad P$ is
a sink of $\Delta$, and \textit{almost left regular} if for any 
$P\in\ind \mathcal{P}(A)$ with $P/\soc P$ lying on $\Delta$, $P/\soc P$ is
a source of $\Delta$.
Moreover, for a finite stable slice $\Delta$, we denote by $M(\Delta)$ the direct sum 
of all modules lying on $\Delta$ and $H(\Delta)=\End_A(M(\Delta))$ the 
endomorphism algebra of $M(\Delta)$.
Then a finite stable slice $\Delta$ of $\Gamma_A$ is said to be \emph{hereditary} 
if the endomorphism algebra $H(\Delta) = \End_A(M(\Delta))$
of $M(\Delta)$ is a hereditary algebra and its valued quiver $Q_{H(\Delta)}$
is the opposite quiver $\Delta^{\op}$ of $\Delta$.

The following theorem is proved in \cite[Theorem~1.1]{SY12} and
extends results established in \cite{SY3} and \cite{SY10} to a general case, 
as explained in the next section.

\begin{theorem}
	\label{th:main}
	Let $A$ be a selfinjective algebra over a field $K$.
	The following statements are equivalent:
	\begin{enumerate}[\upshape (i)]
		\item
		$\Gamma_A$ admits a hereditary almost right regular stable slice.
		\item
		$\Gamma_A$ admits a hereditary almost left regular stable slice.
		\item
		$A$ is socle equivalent to an orbit algebra 
		$\widehat{B}/(\varphi\nu_{\widehat{B}})$,
		where $B$ is a tilted algebra and 
		$\varphi$ is a positive automorphism of $\widehat{B}$.
	\end{enumerate}
\end{theorem}

It should be stressed that, if $K$ is algebraically closed, 
we may replace in {\upshape{(iii)}} ``socle equivalent'' by ``isomorphic'',
but the replacement is not possible in general. 

\begin{proof}[Idea of the proof]
The equivalence of (i) and (ii) is easy, because a selfinjective algebra $A$ satisfies
(i) if and only if $A^{\op}$ satisfies (ii).
Similarly, $A$ satisfies (iii) if and only if
$A^{\op}$ satisfies (iii) in view of the canonical algebra isomorphism
\[
(\wh{B}/(\varphi\nu_{\wh{B}}))^{\op} 
\cong \wh{B^{\op}}/(\psi\nu_{\wh{B^{\op}}})
\]
for a positive automorphism $\psi$ of $\wh{B^{\op}}$.
\medskip

Now we explain 
how to define the algebra $B$ in (iii) under the assumption (i) and conversely
how to find an almost right regular stable slice being hereditary in (i) under the
condition (iii).

Assume that a selfinjective algebra $A$ has a hereditary and almost right regular
stable slice $\Delta$ in the Auslander-Reiten quiver $\Gamma_A$.
Take the direct sum $M$ of all indecomposable modules in $\mod A$ lying on
$\Delta$, and let $I$ be the right annihilator 
$r_A(M)=\{a\in A\mid Ma=0 \}$ of $M$, $B=A/I$ the factor algebra of 
$A$,  and $H=\End_A(M)$.
By the definition of hereditary stable slice, $H$ is a hereditary algebra
and the valued quiver $Q_H$ of $H$ is the opposite quiver $\Delta^{\op}$
of $\Delta$.
Then it is shown 
in \cite[Section~3]{SY12}
that $B$ is a desired tilted algebra and $I$ a deforming ideal 
of $A$
satisfying $r_A(I)=eI$, so that Theorem~\ref{thm:2.5} implies the assertion (iii).

Conversely, assume that a selfinjective algebra $A$ satisfies the assertion (iii).
We shall show how we find a hereditary and almost right regular stable 
slice $\Delta$ in $\Gamma_A$.
Let $A$ be  socle equivalent to 
$\wh{B}/(\varphi\nu_{\wh{B}})$ as in (iii), and
$\Lambda=\wh{B}/(\varphi\nu_{\wh{B}})$.

We observe that, if $\Lambda$ has a stable slice $\Delta$ in 
$\Gamma_{\Lambda}$ being hereditary and almost right regular, then
$A$ has a hereditary almost right regular stable slice, which 
corresponds to $\Delta$ under the given isomorphism 
$A/ \soc (A) \to \Lambda/\soc (\Lambda)$, 
so $A$ may be identified with $\Lambda$,
that is, $A=\wh{B}/(\varphi\nu_{\wh{B}})$.
We consider two cases:

\smallskip

(1) Assume $A$ is of infinite representation type.
Since $B$ is not of finite representation type, by \cite{HW}, \cite{Ho},
the tilted algebra $B$ is not of Dynkin type. 
It follows from general theory developed in \cite{ANS}, \cite{EKS}, \cite{SY3}
and \cite{SY4}
that $\Gamma_A$ admits an acyclic component $\mathcal{C}$ containing
a right stable full translation subquiver $\mathcal{D}$ which is closed
under successors in $\mathcal{C}$ and generalized standard
(see below for definition).
Then any stable slice $\Delta$ in $\mathcal{D}$ is  
a \textit{hereditary} and \textit{right regular} stable slice. 

\smallskip

(2) Assume $A$ is of finite representation type.

\smallskip

\quad  
(a) First consider the case when $A$ is a selfinjective Nakayama algebra.
 Then, for any indecomposable projective module $P$,  
 the module $P/\soc P$  is always the radical of a projective module.
 For an arbitrarily chosen indecomposable projective module $P$, 
 consider the sectional path 
\[
\Delta:\;\; 
\soc P= X_1 \to X_2 \to \cdots \to X_{n-1} \to X_n = \rad P
\]
 of canonical irreducible monomorphisms.
Then $\Delta$ is a \textit{hereditary} and \textit{almost right
regular} stable slice $\Gamma_A$.

\smallskip

\quad  
(b)
Next assume that $A$ is
of finite representation type but not a Nakayama algebra. 
Observe that then 
there exists an indecomposable projective module $P$ such that 
$P/\soc P$ is not the radical of a projective module,
which allows us to take the full subquiver $\Delta_P$ of $\Gamma_A$
whose vertices are $\tau_A^{-1}(P/\soc P)$ and the indecomposable
modules $X$ such that there is a non-trivial sectional path in $\Gamma_A^s$
from $P/\soc P$ to $X$.
It is shown in \cite[Proposition~4.4]{SY12} that 
$\Delta=\tau_A(\Delta_P)$ is a \textit{hereditary}
and \textit{right regular} stable slice \cite[Theorem~3.1]{JPS}.
\end{proof}

We recall some definitions from \cite{SY10}.
Let $A$ be a selfinjective $K$-algebra. 
A stable slice $\Delta$ of $\Gamma_A$
is said to be \textit{semi-regular} if $\Delta$ is left or right regular, and
\textit{regular} if $\Delta$ is left and right regular.
Moreover, a stable slice $\Delta$ is \textit{double $\tau_A$-rigid}
if $\Hom_A(X, \tau_A Y)=0$ and $\Hom_A(\tau_A^{-1} X, Y)=0$
for all indecomposable modules $X$ and $Y$ from $\Delta$.
Note that a double $\tau_A$-rigid stable slice $\Delta$ is finite 
(\cite{S94}, \cite[Lemma VIII.5.3]{ASS})
and a full
valued subquiver of a connected component $\mathcal{C}$ of 
$\Gamma_A^s$ intersecting every $\tau_A$-orbit in $\cC$ exactly once.
The following  theorem proved in \cite[Theorem~1]{SY10}
is a special case of Theorem \ref{th:main}.
\begin{theorem} \label{double rigid}
	Let $A$ be a basic, connected, selfinjective algebra over a field $K$.
	The following statements are equivalent:
	\begin{enumerate}[\upshape (i)]
		\item $\Gamma_A$ admits a semi-regular double $\tau_A$-rigid stable
		slice.
		\item $A$ has one of the following forms:
		\begin{enumerate}[\upshape (a)]
			\item $A \cong \wh{B}/(\varphi\nu_{\wh{B}})$ as algebras, where
			$B=\End_H(T)$ for a hereditary algebra $H$ and a tilting module $T$
			in $\mod H$ either without nonzero projective direct summand
			or without nonzero injective direct summand, and $\varphi$ is
			a strictly positive automorphism of $\wh{B}$.
			\item $A$ is socle equivalent to 
			$\wh{B}/(\varphi\nu_{\wh{B}})$, where
			$B=\End_H(T)$ for a hereditary algebra $H$ and a tilting module $T$
			in $\mod H$ without nonzero projective 
			or injective direct summands, 
			and $\varphi$ is a rigid automorphism of $\wh{B}$.
		\end{enumerate}
	\end{enumerate}
		Moreover, if $K$ is an algebraically closed field, 
we may replace in {\upshape{(ii)(b)}} ``socle equivalent'' by ``isomorphic''. 
\end{theorem}

\section{Selfinjective orbit algebras of infinite representation type}
\label{sec:5}

The first study of the socle deformations of a selfinjective orbit algebra
over an arbitrary field $K$
was given in \cite{SY1}, \cite{SY3}, 
where representation-infinite selfinjective algebras are mainly considered. 
In this section, we present two theorems on representation-infinite
socle deformations of a selfinjective orbit algebra of a tilted type.
	
Following \cite{S2}, a full translation subquiver $\Sigma$ 
of the Auslander-Reiten quiver $\Gamma_{\Lambda}$
of an algebra $\Lambda$
is said to be \textit{generalized standard} if 
$\bigcap_{m>0} \rad^m \Sigma =0$.
The following theorem is from \cite{SY1} and \cite{SY3}.
	
\begin{theorem}
Let $A$ be a basic, connected, selfinjective algebra over a field $K$. 
The following statements are equivalent:
\begin{enumerate}[\upshape(i)]
 \item
  $\Gamma_A$ admits an acyclic generalized standard right stable
  full translation subquiver $\Sigma$ which is closed 
  under successors in $\Gamma_A$.
 \item
  $\Gamma_A$ admits an acyclic generalized standard left stable
  full translation subquiver $\Omega$ which is closed 
  under predecessors in $\Gamma_A$.
 \item $A$
  is socle equivalent to an orbit algebra 
  $\wh{B}/(\varphi\nu_{\wh{B}})$, where  $B$ is a tilted algebra
  not of Dynkin type 
  and $\varphi$ is a positive
  automorphism of $\wh{B}$.
\end{enumerate}
Moreover, if $K$ is an algebraically closed field, 
we may replace in {\upshape{(iii)}} ``socle equivalent'' by ``isomorphic''. 
\end{theorem}
	
	In order to apply Theorem \ref{th:main}, let $\Sigma$ be the generalized
	standard translation subquiver given in the statement (i), and let
	$I$ be the annihilator $\ann_A \Sigma$ of $\Sigma$ in $A$ and $B=A/I$.
	Then it is shown 
	in \cite{SY1}
	that $I$ is a deforming ideal of $A$ such that 
	\[
	r_A(I)= eI \quad \text{and} \quad l_A(I)=Ie ,
	\]
	for a residual identity $e$ of $B$.
	It should be noted that $\ann_A(\Sigma)$ is the same as the annihilator
	in $A$ of any section in $\Sigma$.
	\medskip

	The following theorem from \cite[Theorem 2]{SY10} 
	shows another characterization of the socle deformations
	of a representation-infinite orbit algebra 
	$\wh{B}/(\varphi\nu_{\wh{B}})$ of tilted type,
and is a consequence of Theorem~\ref{double rigid}.
	
\begin{theorem}
\label{infinite rep}
Let $A$ be a basic, connected, selfinjective
algebra of infinite representation type over a field $K$.
The following statements are equivalent:
\begin{enumerate}[\upshape (i)]
 \item
	$\Gamma_A$ admits a regular double $\tau_A$-rigid stable slice.
 \item
	$A$ is socle equivalent to an orbit algebra 
	$\widehat{B}/(\varphi\nu_{\widehat{B}})$,
	where $B$ is a tilted algebra not of Dynkin type and 
	$\varphi$ is a positive automorphism of $\widehat{B}$.
\end{enumerate}
Moreover, if $K$ is an algebraically closed field, 
we may replace in {\upshape{(ii)}} ``socle equivalent'' by ``isomorphic''. 
\end{theorem}

\section{Selfinjective orbit algebras of finite representation type}
\label{sec:6}

In this section, we present recent results concerning
the structure of selfinjective algebras of finite representation type,
which are socle equivalent to selfinjective orbit algebras
of tilted algebras of Dynkin type.

We first recall the following old result proved by
C.~Riedtmann \cite{Ried} and G.~Todorov \cite{T}
describing the structure of the stable Auslander-Reiten
quiver of a selfinjective algebra of finite representation
type (see \cite[Theorem~IV.15.6]{SY9} for a proof).

\begin{theorem}
Let $A$ be a non-simple, basic, connected, selfinjective 
algebra of finite representation type over a field $K$.
Then the stable Auslander-Reiten quiver $\Gamma_A^s$ of $A$
is isomorphic to the orbit valued translation quiver 
$\bZ \Delta / G$, where $\Delta$ is a Dynkin quiver
of type 
$\bA_n(n\geq1)$,
$\bB_n(n\geq2)$,
$\bC_n(n\geq3)$,
$\bD_n(n\geq4)$,
$\bE_6$,
$\bE_7$,
$\bE_8$,
$\bF_4$
and
$\bG_2$,
and $G$ is an admissible infinite cyclic group of automorphisms 
of $\bZ \Delta$.
\end{theorem}

We note that if $A = \wh{B}/(\varphi)$ is an orbit algebra 
with $B$ a tilted algebra of Dynkin type $\Delta$
over a field $K$
and $\varphi$ is a strictly positive automorphism of $\wh{B}$
then $\Gamma_A^s \cong \bZ \Delta / G$, where $G$ is the
infinite cyclic group of automorphisms of $\bZ \Delta$
induced by $\varphi$.
It would be interesting to know when a selfinjective algebra $A'$
over an arbitrary field $K$
is socle equivalent to 
a
selfinjective orbit algebra
$A = \wh{B}/(\varphi)$ of Dynkin type.
It follows from the classification result of C. Riedtmann,
presented in Section~\ref{sec:2}, that all representation-finite
selfinjective algebras over an algebraically closed field have this property.
But for an arbitrary field $K$, this seems to be a hard problem.

We consider first selfinjective Nakayama algebras.
A module $X$ in a module category $\mod A$ is said to be
\emph{composition free} if all simple composition factors
of $X$ occur with multiplicity one.
The following result is a consequence of Theorem \ref{th:main}.

\begin{theorem}
\label{th:6.2}
Let $A$ be a non-simple, basic, connected, selfinjective Nakayama algebra
over a field $K$.
Then the following statements are equivalent:
\begin{enumerate}[\upshape(i)]
 \item
  Any indecomposable projective module $P$ in $\mod A$ 
  has composition free radical $\rad P$.
 \item
  Any indecomposable projective module $P$  in $\mod A$ 
  has composition free socle factor $P/\soc P$.
 \item
  $A$ is socle equivalent to an orbit algebra 
  $\wh{H}/(\varphi\nu_{\wh{H}})$, where
  $H$ is a hereditary Nakayama algebra
  and $\varphi$ is a positive automorphism of $\wh{H}$.
 \item
  $A$ is socle equivalent to an orbit algebra 
  $\wh{B}/(\varphi\nu_{\wh{B}})$, where
  $B$ is an algebra  and $\varphi$ is a positive automorphism
  of $\wh{B}$.		
\end{enumerate}
\end{theorem}

\begin{proof}
The equivalence (i) $\Leftrightarrow$ (ii) and the implication (iii) $\Rightarrow$ (iv) are trivial. 

Assume that (i) holds. 
Let $P$ be an indecomposable projective module in $\mod A$.
Since $A$ is Nakayama, we have in $\Gamma_A$ the sectional path 
\[
	\Delta: \;\; \soc P = X_{1} \to X_2 \to \cdots \to X_{n-1} \to X_n = \rad P,
\]
given by the irreducible inclusion monomorphisms, which
is an almost right regular stable slice of $\Gamma_A$.
Moreover, the slice $\Delta$ is hereditary, because $\rad P$ 
is composition free.
Therefore, applying arguments from the proof of Theorem \ref{th:main},
we conclude that $A$ is socle equivalent to an orbit algebra 
$\wh{H}/(\varphi\nu_{\wh{H}})$ for 
a hereditary Nakayama algebra $H$ with the Gabriel quiver
$Q_H = \Delta = \Delta^{\op}$ and
a positive automorphism 
$\varphi$ of $\wh{H}$.
Hence  (iii) is satisfied.

Assume that (iii) holds. 
Let $H$ be a hereditary Nakayama algebra
and $\varphi$ is a positive automorphism of $\wh{H}$
such that the orbit algebra
$\Lambda = \wh{H}/(\varphi\nu_{\wh{H}})$
is socle equivalent to $A$.
Without loss of generality, we may assume that 
the quotient algebras 
$\Lambda/\soc(\Lambda)$ 
and $A/\soc(A)$ 
are equal.
Let $1_H=e_1+ \cdots + e_n$ be a decomposition 
of the identity of $H$ into a sum of pairwise orthogonal
primitive idempotents.
Then it follows from the definition of $\wh{H}$ 
that every indecomposable projective $\wh{H}$-module
$e_{m,i} \wh{H}$, $(m,i) \in \mathbb{Z} \times \{1, \dots, n\}$,
is uniserial.
Moreover, for any $(m,i)$ and $(r,j)$ in $\mathbb{Z} \times \{1, \dots, n\}$
with $m \neq r$,
the socle factors 
$e_{m,i}\wh{H} / \soc(e_{m,i}\wh{H})$
and
$e_{r,j}\wh{H} / \soc(e_{r,j}\wh{H})$
have disjoint simple composition factors.
Then it follows that any indecomposable
projective module $P$ in $\mod \Lambda$
has composition free socle factor $P/\soc P$.
Since the socle factors of indecomposable
projective modules in $\mod \Lambda$ and $\mod A$
are exactly the indecomposable projective modules
in $\mod \Lambda / \soc(\Lambda) = \mod A / \soc(A)$,
we conclude that $A$ satisfies (ii).

Assume now that (iv) holds, that is,
$A$ is socle equivalent to an orbit algebra
$\Lambda = \wh{B}/(\varphi\nu_{\wh{B}})$, where
$B$ is an algebra  and $\varphi$ is a positive automorphism
of $\wh{B}$.		
Clearly, such an algebra $B$ is basic and connected.
We claim that $B$ is a hereditary Nakayama algebra,
and hence $A$ satisfies (iii).
Since $A$ is a Nakayama algebra and
$A / \soc(A) \cong \Lambda / \soc(\Lambda)$,
we conclude that 
$\Lambda$ is a non-simple selfinjective Nakayama algebra.
We also note that $B$ is a quotient algebra of $\Lambda$,
because $\Lambda$ is an orbit algebra of $\wh{B}$
with respect to $(\varphi\nu_{\wh{B}})$ with $\varphi$ positive.
Therefore, we obtain that $B$ is a Nakayama algebra.
Hence, we have only to show that any non-simple projective 
module in $\mod B$ has projective socle. 
Let $1_B=e_1+ \cdots + e_n$ be a decomposition 
of the identity of $B$ into a sum of pairwise orthogonal
primitive idempotents.
Remember that $e_{m,i}$, $(m,i) \in \mathbb{Z} \times \{1, \dots, n\}$,
are pairwise orthogonal primitive idempotents of $\wh{B}$ and
$e_{m,i}\wh{B}= e_i B \oplus D(Be_i)$.
For convenience, we write $B_m$ for $1_m\wh{B}1_m$ where
$1_m=e_{m,1}+ \cdots + e_{m,n}$, and $\mathbb{D}(\wh{B})
= \bigoplus_{m\in \mathbb{Z}} D(B_m)$.
Note that $B_m$ is a copy of $B$, and observe that
\[
	\rad \wh{B}= \bigg(\bigoplus_{m\in \mathbb{Z}} \rad B_m\bigg)
	\oplus \mathbb{D}(\wh{B}).
\]
Since $\Lambda = \wh{B}/(\varphi\nu_{\wh{B}})$ is a Nakayama algebra, 
each $e_{m,i}\wh{B}$ and
so $e_iB$ is uniserial for all $m\in \mathbb{Z}$, $i\in \{1, \dots, n\}$.
Now let $P$ be a non-simple indecomposable projective module in $\mod B$,
and assume $P=e_1B$, without loss of generality. Let 
\[
    \soc P \cong \tp(e_tB)
\]
for some $t\in \{1, \dots, n\}$.
Take the $\wh{B}$-submodule
$
      M = \soc (e_1B) \oplus D(Be_1)
$
of $e_{0,1}\wh{B}= e_1B \oplus D(Be_1)$.
Then $ M(\rad \wh{B}) = D(Be_1)$,
because $M/D(Be_1)$ is a simple $\wh{B}$-module and $M$ is uniserial.
Let $f: e_{0,t}\wh{B} \to e_{0,1}\wh{B}$ be a  
$\wh{B}$-homomorphism with $f(e_{0,t}\wh{B})=M$.
It is shown that
\[
   M \cdot \mathbb{D}(\wh{B})=D(Be_1)
\]
taking into account 
the 
uniseriality of $e_{0,t}\wh{B}$.
On the other hand,
$f(e_{0,t}\wh{B}\cdot \rad \wh{B})= M(\rad \wh{B})$ and
$f(e_{0,t}\wh{B}\cdot \mathbb{D}(\wh{B}))
   = M \cdot \mathbb{D}(\wh{B})$.
Consequently, we have 
$f(e_{0,t}\wh{B}\cdot \rad \wh{B})
   =f(e_{0,t}\wh{B}\cdot \mathbb{D}(\wh{B}))$, 
which implies 
\[
   e_{0,t}\wh{B}\cdot \rad \wh{B}
   =e_{0,t}\wh{B}\cdot \mathbb{D}(\wh{B}) + \Ker f
   =e_{0,t}\wh{B}\cdot \mathbb{D}(\wh{B}) .
\]
Indeed, since 
$f(e_{0,t}\wh{B}\cdot \mathbb{D}(\wh{B}))
   =D(Be_1)\neq 0$ 
and $e_{0,t}\wh{B}$ is uniserial, $\Ker f$ is 
contained in $e_{0,t}\wh{B}\cdot \mathbb{D}(\wh{B})$.
As a result, we conclude that 
\[\tp (e_tB)=
   e_{0,t}\wh{B}/(e_{0,t}\wh{B}\cdot \rad \wh{B})
   = e_{0,t}\wh{B}/(e_{0,t}\wh{B}\cdot \mathbb{D}(\wh{B})) = e_tB,
\]
and hence $\soc P$ is projective in $\mod B$, as desired.
\end{proof}

We obtain also the following consequence of Theorems
 \ref{thm:3.7} and \ref{th:6.2}.

\begin{corollary}
\label{cor:6.3}
Let $A$ be a non-simple, basic, connected, selfinjective Nakayama algebra
over a field $K$.
Then the following statements are equivalent:
\begin{enumerate}[\upshape(i)]
 \item
  Any indecomposable projective module $P$ in $\mod A$ 
  is composition free.
 \item
  $A$ is isomorphic to an orbit algebra 
  $\wh{H}/(\varphi\nu_{\wh{H}})$, where
  $H$ is a hereditary Nakayama algebra
  and $\varphi$ is a strictly positive automorphism of $\wh{H}$.
 \item
  $A$ is isomorphic to an orbit algebra 
  $\wh{B}/(\varphi\nu_{\wh{B}})$, where
  $B$ is an algebra  and $\varphi$ is a strictly 
  positive automorphism
  of $\wh{B}$.		
\end{enumerate}
\end{corollary}

Following \cite{RSS} a \emph{short cycle}
in a module category $\mod A$
is a sequence $X \to Y \to X$
of non-isomorphisms between two indecomposable modules
$X$ and $Y$ in $\mod A$.
It was shown in \cite[Corollary~2.2]{RSS}
that if an indecomposable module $M$
does not lie on a short cycle in $\mod A$
then $M$ is uniquely determined 
(up to isomorphism) by its composition factors,
that is, by its image $[M]$ in the Grothendieck 
group $K_0(A)$.
Moreover, it is known that if $\mod A$ has no short cycles, 
then $A$ is of finite representation type \cite{HL}.
We have also the following fact proved in
\cite[Lemma~3.2]{JPS}.

\begin{lemma} 
\label{JPS:lem}
Let $A$ be a selfinjective algebra which 
does not admit a short cycle in $\mod A$.
Then all stable slices in $\Gamma_A$ are
double $\tau_A$-rigid.
\end{lemma}

We also recall the following fact proved in 
\cite[Theorem 3.1]{JPS}.

\begin{theorem} 
\label{JPS:non-Nakayama}
Let $A$ be a non-simple, connected, selfinjective algebra
of finite representation type over a field $K$.
The following statements are equivalent:
\begin{enumerate}[\upshape(i)]
 \item
  $\Gamma_A$ admits a semi-regular stable slice.
 \item
  $A$ is not a Nakayama algebra.
\end{enumerate}
\end{theorem}

The following theorem
was proved by A.~Jaworska-Pastuszak and A.~Skowro\'nski
in \cite{JPS}.



\begin{theorem}
Let $A$ be a non-simple, basic, connected, selfinjective algebra
over a field $K$.
The following statements are equivalent:
\begin{enumerate}[\upshape(i)]
 \item 
  $\mod A$ has no short cycles.
 \item
  $A$ is isomorphic to an orbit algebra 
  $\wh{B}/(\psi\nu_{\wh{B}}^2)$, where
  $B$ is a tilted algebra of Dynkin type and $\psi$ is a strictly
  positive automorphism of $\wh{B}$.
\end{enumerate}
\end{theorem}

It was shown in \cite[Theorem~3.7]{JPS} that the equivalence of
(i) and (ii) holds if $A = \wh{B}/G$ for a tilted algebra $B$
of Dynkin type and an admissible group $G$ of automorphisms
of $\wh{B}$, applying the fact that the push-down functor
$F_{\lambda} : \mod \wh{B} \to \mod A$,
induced by the canonical Galois covering functor
$F : \wh{B} \to \wh{B}/G = A$, 
is a Galois covering of module categories.
We explain now why (i) implies (ii).
Therefore, assume that $\mod A$ has no short cycles.
By the above comments, it is enough to show that
$A$ is isomorphic to an orbit algebra 
$\wh{B}/(\varphi \nu_{\wh{B}})$
for a tilted algebra $B$ of Dynkin type and a strictly
positive automorphism $\varphi$ of $\wh{B}$.
If $A$ is not a Nakayama algebra, this follows
from 
Theorems \ref{double rigid} and \ref{JPS:non-Nakayama}, 
and Lemma~\ref{JPS:lem}.
Assume $A$ is a Nakayama algebra.
Since every indecomposable projective module $P$ in $\mod A$
does not lie on a short cycle,
we easily conclude that $P$ is composition free,
and then $A$ has a required form
$\wh{B}/(\varphi \nu_{\wh{B}})$
by Theorem~\ref{th:6.2}. 

In \cite{BS1} M. B\l aszkiewicz \and A. Skowro\'nski 
investigated the structure of selfinjective algebras
of finite representation type
whose module category admits maximal almost split sequences.
For an algebra $A$ and an almost split sequence
\[
  0  \longrightarrow \tau_A X \longrightarrow 
  Y \longrightarrow X \longrightarrow 0
\]
in $\mod A$, we may consider the numerical invariant 
$\alpha(X)$ of $X$ being the number of summands
in a decomposition 
$Y = Y_1 \oplus \dots \oplus Y_r$ of $Y$
into a direct sum of indecomposable modules.
Then $\alpha(X)$ measures the complexity of homomorphisms
in $\mod A$ with domain $\tau_A X$ and codomain $X$.
It has been proved by R.~Bautista and S.~Brenner in \cite{BB}
(see also \cite{L} for an alternative proof)
that if $A$ is of finite
representation type and $X$ is an indecomposable
nonprojective module in $\mod A$, then $\alpha(X) \leq 4$,
and if $\alpha(X) = 4$, then the middle term $Y$
of an almost split sequence in $\mod A$ with the right term $X$
admits an indecomposable projective-injective direct summand.
Recall that (by general theory) if $P$ is an indecomposable
projective-injective module in $\mod A$, then there is in $\mod A$
 an almost split sequence of the form
\[
  0  \longrightarrow \rad P \longrightarrow 
  (\rad P / \soc P) \oplus P \longrightarrow P/\soc P \longrightarrow 0 .
\]
If $A$ is of finite representation type, then such an almost
split sequence with $\alpha(P/\soc P) = 4$
is called a \emph{maximal almost split sequence} in $\mod A$.

Let $A$ be a selfinjective algebra of finite representation type,
and assume that $\mod A$ admits an indecomposable
projective-injective module $P$ with $\alpha(P/\soc P) = 4$.
We denote by $\Delta_P$ the full valued subquiver of $\Gamma_A$
given by the module $\tau_A^{-1}(P/\soc P)$
and all indecomposable modules in $\mod A$ such that
there is a non-trivial sectional path in $\Gamma_A$ from
$P/\soc P$ to $X$.
It was shown in \cite[Theorem~5.2]{BS1}
that $\Delta_P$ is a Dynkin quiver such that
$\Gamma_A^s = \bZ \Delta_P/G$ for an admissible
automorphism group $G$ of the translation quiver
$\bZ \Delta_P$.
We denote by $M_P$ the direct sum of all modules lying
on $\Delta_P$.
Then the main theorem in \cite[Theorems 1 and 2]{BS1} can be formulated as follows.
 
\begin{theorem}
Let $A$ be a basic, connected, selfinjective algebra over a field $K$. 
Then the following statements are equivalent:
\begin{enumerate} [\upshape(i)]
 \item
  $A$ is of finite representation type and admits an indecomposable
  projective module $P$ with
  $\alpha(P/ \soc P)=4$
  and
  $\Hom_A(M_P,\tau_A M_P) = 0$.
 \item
  $A$ is socle equivalent to an orbit algebra $\wh{B}/(\varphi\nu_{\wh{B}}^m)$
  for some positive integer $m$, 
  a rigid automorphism $\varphi$ of $\wh{B}$,
  and a tilted algebra $B$ of Dynkin type having an
  indecomposable projective module $Q$ whose top
  is injective and the radical is a direct sum of three
  indecomposable projective modules.
\end{enumerate}
\end{theorem}

We note that, if $K$ is algebraically closed, 
then the 
condition
$\Hom_A(M_P,\tau_A M_P) = 0$ in (i) 
is superfluous, and we may replace in (ii) 
``socle equivalent'' by ``isomorphic''. 
We refer to \cite[Section~3]{BS1} for a complete
description of the tilted algebras $B$ of Dynkin type
occurring in (ii).

\section{Stable equivalences of selfinjective orbit algebras}
\label{sec:7}

We end this article with the following combination of 
\cite[Theorem~1]{SY2} and \cite[Theorem~1]{SY7},
and its consequence.

\begin{theorem}
Let $A$ be a non-simple, basic, connected, selfinjective algebra 
over a field $K$. 
Then the following statements are equivalent:
\begin{enumerate} [\upshape(i)]
 \item
  $A$ is stably equivalent to an orbit algebra $\wh{B}/(\varphi\nu_{\wh{B}})$,
  where $B$ is a tilted algebra and
  $\varphi$ is a positive automorphism of $\wh{B}$.
 \item
  $A$ is socle equivalent to an orbit algebra $\wh{B}/(\varphi\nu_{\wh{B}})$,
  where $B$ is a tilted algebra and
  $\varphi$ is a positive automorphism of $\wh{B}$.
\end{enumerate}
Moreover, if $K$ is an algebraically closed field, 
we may replace in {\upshape{(ii)}} ``socle equivalent'' by ``isomorphic''. 
\end{theorem}

\begin{corollary}
\label{cor:6.3}
Let $K$ be an algebraically closed field. 
Then the class of orbit algebras $\wh{B}/(\varphi\nu_{\wh{B}})$,
where $B$ is a tilted algebra over $K$ and
$\varphi$ is a positive automorphism of $\wh{B}$,
is closed under stable equivalences.
\end{corollary}

\bigskip

\begin{center}
\scshape{Acknowledgements}\label{ackref}
\end{center}
	\medskip
	
The authors were supported by the research grant 
DEC-2011/02/A/ST1/00216 of the National Science Center Poland.
The second named author was also supported by
JSPS KAKENHI grants 16K05091 and 19K03417.

\bigskip
\bigskip
\noindent
   Andrzej Skowro\'nski\\
   Faculty of Mathematics and Computer Science\\
   Nicolaus Copernicus University\\
   Chopina~12/18\\
   87-100 Toru\'n\\
   Poland\\
   \email{skowron@mat.uni.torun.pl}

\bigskip

\noindent
   Kunio Yamagata\\
   Department of Mathematics\\
    Tokyo University of Agriculture and Technology\\
    Nakacho 2-24-16, Koganei\\
    Tokyo 184-8588\\
   Japan\\
   \email{yamagata@cc.tuat.ac.jp}

\begin{thebibliography}{99}
%

%
\bibitem{AKMW} 
{S. Ariki, R. Kase, K. Miyamoto \and K. Wada},
`Self-injective cellular algebras whose representation type are tame of polynomial growth',
{\em Algebr. Represent. Theory}, in press, doi:10.1007/s10468-019-09872-w.
%
\bibitem{ANS}
{I. Assem, J. Nehring \and A. Skowro\'nski},
`Domestic trivial extensions of simply connected algebras',
{\em Tsukuba J. Math. }150  (2014) 415--452.
\bibitem{ASS}
 {I. Assem, D. Simson \and A. Skowro\'nski}, 
 {\em Elements of the Representation Theory of Associative Algebras 1: Techniques of Representation Theory}, 
 London Math. Soc. Stud. Texts {65} 
 (Cambridge Univ. Press, Cambridge, 2006).
%
%
\bibitem{BB}
 {R. Bautista \and S. Brenner}, 
    `On the number of terms in the middle of an almost split sequence', 
    {\em Representations of Algebras}, 
    Lecture Notes in Math. 903, 
    (Springer-Verlag, Berlin-Heidelberg, 1981) 1--8.
%
\bibitem{BES}
  {J. Bia\l kowski, K. Erdmann \and A. Skowro\'nski},
  `Periodicity of self-injective algebras of polynomial growth',
{\em  J. Algebra  } 443 (2015)  200--269. 
%
\bibitem{BiS1}
  {J. Bia\l kowski \and A. Skowro\'nski},
  `Selfinjective algebras of tubular type',
{\em  Colloq. Math.  }94  (2002)  175--194. 
%
\bibitem{BiS2}
  {J. Bia\l kowski \and A. Skowro\'nski},
  `On tame weakly symmetric algebras having only periodic modules',
{\em   Arch. Math. (Basel)  }81  (2003)  142--154. 
%
\bibitem{BiS3}
  {J. Bia\l kowski \and A. Skowro\'nski},
  `Socle deformations of selfinjective algebras of tubular type',
{\em  J. Math. Soc. Japan  }56  (2004)  687--716. 
%
\bibitem {BS1}
 {M. B\l aszkiewicz \and A. Skowro\'nski}, 
 `On selfinjective algebras of finite representation type with maximal almost split sequences', 
 {\em J. Algebra  }422  (2015) 450--486.
%
\bibitem{BoS1}
  {R. Bocian \and A. Skowro\'nski},
  `Symmetric special biserial algebras of Euclidean type',
{\em  Colloq. Math.  }96  (2003)  121--148. 
%
\bibitem{BoS2}
  {R. Bocian \and A. Skowro\'nski},
  `Weakly symmetric algebras of Euclidean type',
{\em J. reine angew. Math.  }580  (2005)  157--199. 
%
\bibitem{BoS3}
  {R. Bocian \and A. Skowro\'nski},
  `Socle deformations of selfinjective algebras of Euclidean type.',
{\em Comm. Algebra  }34  (2006)  4235--4257. 
%
%
\bibitem {BLR}
 {O. Bretscher,  C. L\"aser \and C. Riedtmann}, 
 `Self-injective and simply connected algebras', 
 {\em Manuscripta Math.  }36  (1981/82) 253--307.
%
\bibitem{DS1}
 {P. Dowbor \and A. Skowro\'nski}, 
`On Galois coverings of tame algebras',
 {\em Arch. Math. (Basel)  }44  (1985)  522--5297. 
%
\bibitem{DS}
 {P. Dowbor \and A. Skowro\'nski}, 
`Galois coverings of representation-infinite algebras',
 {\em Comment. Math. Helv.  }62  (1987)  311--337. 
%
\bibitem{Du}
{A.~Dugas},
`Periodic resolutions and self-injective algebras of finite type',
{\em J. Pure Appl. Algebra}  {214}  (2010)   990--1000.
%
\bibitem{EKS}
  {K. Erdmann, O. Kerner \and A. Skowro\'nski},
  `Self-injective algebras of wild tilted type',
  {\em  J. Pure Appl. Algebra  }149  (2000)   127--176.
\bibitem{ES}
  {K. Erdmann \and A. Skowro\'nski},
  `On Auslander-Reiten components of blocks and selfinjective biserial algebras'
  {\em  Trans. Amer. Math. Soc. }330 (1992) 165--189. 
%
\bibitem{FS1}
  {R. Farnsteiner \and A. Skowro\'nski},
  `Classification of restricted Lie algebras with tame principal block',
  {\em  J. reine angew. Math.  }546  (2002) 1--45.
%
\bibitem{FS2}
  {R. Farnsteiner \and A. Skowro\'nski},
  `The tame infinitesimal groups of odd characteristic',
  {\em  Adv. Math.  }205  (2006) 229--274.
%
\bibitem {G}
    P. Gabriel, 
    `The universal cover of a representation-finite algebra', 
    {\em Representations of Algebras}, 
    Lecture Notes in Math. 903, 
    (Springer-Verlag, Berlin-Heidelberg, 1981) 68--105.
%
%
\bibitem {HL}
{D. Happel \and S. Liu}, 
`Module categories without short cycles are of finite type', 
{\em Proc. Amer. Math. Soc. }{120} (1994) 371--375.
%
%
\bibitem{Ho}
   {M. Hoshino}, 
  `Trivial extensions of tilted algebras', 
  {\em Comm. Algebra }{10} (1982) 1965--1999.
%
\bibitem{HW}
  {D. Hughes \and J. Waschb\"usch}, 
  `Trivial extensions of tilted algebras', 
  {\em Proc. London Math. Soc. }{46} (1983) 347--364.
%
%
\bibitem{JPS}
{A. Jaworska-Pastuszak \and A. Skowro\'nski},
`Selfinjective algebras without short cycles of indecomposable modules',
{\em J. Pure Appl. Algebra }222  (2018)  3432--3447.
%
%
\bibitem{KS}
{ O. Kerner \and A. Skowro\'nski},
`On module categories with nilpotent infinite radical',
{\em Compos. Math.  }{77}  (1991)  313--333. 
%
\bibitem{LS1}
  {H. Lenzing \and A. Skowro\'nski},
  `On selfinjective algebras of Euclidean type',
{\em  Colloq. Math.  }79  (1999)  71--76. 

\bibitem{LS2}
  {H. Lenzing \and A. Skowro\'nski},
  `Roots of Nakayama and Auslander-Reiten translations',
{\em  Colloq. Math.  }86  (2000)  209--230. 

\bibitem{LS3}
  {H. Lenzing \and A. Skowro\'nski},
  `Selfinjective algebras of wild canonical type',
{\em  Colloq. Math.  }96  (2003)  245–275. 
%
%
\bibitem{L}
{S. Liu},
`Almost split sequences for nonregular modules',
{\em Fund. Math. }143 (1993) 183--190.
%
%
\bibitem{NS}
{J. Nehring \and A. Skowro\'nski},
`Polynomial growth trivial extensions of simply connected algebras',
{\em  Fund. Math.  }132  (1989) 117--134.
%
\bibitem{O}
{Y. Ohnuki},
`Stable equivalence induced by a socle equivalence',
{\em Osaka J. Math. Soc. }39 (2002), 259--266.
%
\bibitem{PS}
{Z. Pogorza\l y \and A. Skowro\'nski},
`Selfinjective biserial standard algebras',
{\em J. Algebra }138  (1991) 491--504.
%
\bibitem{RSS}
{I. Reiten, A. Skowro\'nski \and  S. O. Smal\o}, 
`Short chains and short cycles of modules',
{\em Proc. Amer. Math. Soc. }117 (1993) 343--354.
 
%



\bibitem{Ried}
{C. Riedtmann},
`Algebren, Darstellungsk\"ocerm, \"Uberlagerungen and zur\"uk', 
{\em Comment. Math. Helv. }55 (1980) 199--224.


\bibitem{Rd2}
{C. Riedtmann},
 `Representation-finite self-injective algebras of class $\bA_n$', 
  {\em Representation Theory, II}, 
  Lecture Notes in Math. 832, 
 (Springer-Verlag, Berlin-Heidelberg, 1980) 449--520.

\bibitem{Rd3}
{C. Riedtmann},
`Representation-finite self-injective algebras of class $\bD_n$', 
{\em  Compositio Math. }49 (1983) 199--224.

\bibitem{S3}
{A. Skowro\'nski},
`Selfinjective algebras of polynomial growth',
{\em Math. Ann. }285  (1989) 177--199.
\bibitem{S2}
{A. Skowro\'nski},
`Generalized standard Auslander-Reiten components',
{\em J. Math. Soc. Japan }46 (1994) 517--543.
\bibitem{S94}
{A. Skowro\'nski},
`Regular Auslander-Reiten components containing directing modules',
{\em Proc. Amer. Math. Soc. }120 (1994) 19--26.
%
\bibitem{S4}
{A. Skowro\'nski},
`Selfinjective algebras: finite and tame type',
   {\em Trends in Representation Theory of Algebras and Related Topics},
   Contemporary Math. 406, 
   (Amer. Math. Soc., Providence, RI, 2006) 169--238.
%
\bibitem{SY1}
{A. Skowro\'nski \and K. Yamagata}, 
`Socle deformations of selfinjective algebras', 
{\em Proc. London Math. Soc. }{72} (1996) 545--566.
%
\bibitem{SY2}
{A. Skowro\'nski \and K. Yamagata}, 
`Stable equivalences of selfinjective algebras of tilted type',
{\em Arch. Math. (Basel) }70 (1998) 341--350.
%
\bibitem{SY3}
{A. Skowro\'nski \and K. Yamagata}, 
`Galois coverings of selfinjective algebras by repetitive algebras', 
{\em Trans. Amer. Math. Soc. }{351} (1999) 715--734.
%
\bibitem{SY4}
{A. Skowro\'nski \and K. Yamagata}, 
`On selfinjective artin algebras having nonperiodic generalized standard Auslander-Reiten components', 
{\em Colloq. Math. }{96} (2003) 235--244.
\bibitem{SY5}
{A. Skowro\'nski \and K. Yamagata}, 
`On invariability of selfinjective algebras of tilted type under stable equivalences', 
{\em Proc. Amer. Math. Soc. }{132} (2004) 659--667.
%
\bibitem{SY6}
{A. Skowro\'nski \and K. Yamagata}, 
`Positive Galois coverings of selfinjective algebras', 
{\em Adv. Math. }{194} (2005) 398--436.
%
\bibitem{SY7}
{A. Skowro\'nski \and K. Yamagata}, 
`Stable equivalences of selfinjective artin algebras of Dynkin type', 
{\em Algebr. Represent. Theory }{9} (2006) 33--45.
%
\bibitem{SY8}
{A. Skowro\'nski \and K. Yamagata}, 
`Selfinjective algebras of quasitilted type', 
{\em  Trends in Representation Theory of Algebras and Related Topics}, 
    European Math. Soc. Series of Congress Reports, 
    (European Math. Soc. Publ. House, Z\"urich, 2008) 639--708.
%
\bibitem{SY9}
{A. Skowro\'nski \and K. Yamagata}, 
{\em Frobenius Algebras I.
    Basic Representation Theory},
    EMS Textbooks in Mathematics,
    (European Math. Soc. Publ. House, Z\"urich, 2011).
%
\bibitem{SY10}
{A. Skowro\'nski \and K. Yamagata}, 
`On selfinjective algebras of tilted type', 
{\em Colloq. Math. }{141} (2015) 89--117.
%
\bibitem{SY11}
{A. Skowro\'nski \and K. Yamagata}, 
{\em Frobenius Algebras II.
    Tilted and Hochschild Extension Algebras},
    EMS Textbooks in Mathematics,
    (European Math. Soc. Publ. House, Z\"urich, 2017).
%
\bibitem{SY12}
{A. Skowro\'nski \and K. Yamagata}, 
`Selfinjective algebras with hereditary stable slice',
{\em J. Algebra} {530} (2019) 146--162.

 \bibitem {T}
  {G. Todorov}, 
 `Almost split sequences for $\Tr D$-periodic modules', 
  {\em Representation Theory, II}, 
  Lecture Notes in Math. 832, 
 (Springer-Verlag, Berlin-Heidelberg, 1980) 600--631.

\end{thebibliography}
\end{document}